\documentclass[12pt]{article}
\usepackage[russian,english]{babel}	

\usepackage{graphicx}													
\usepackage{amsmath}
\usepackage{amsthm}						
\usepackage{amsfonts}								
\usepackage{amssymb}
\usepackage{hyperref}

\usepackage{ucs}
\usepackage{tikz}
\usepackage{latexsym}
\usepackage{graphicx}
\usepackage{wrapfig}
\usepackage{caption}
\usepackage{subcaption}
\usepackage{indentfirst}
\usepackage[left=2.5cm,right=2.5cm,top=2.4cm,bottom=2.4cm,bindingoffset=0cm]{geometry}
\usepackage{enumerate}
\usepackage{makecell}	
\usepackage{tikz}

\newtheorem{lemma}{Lemma}[section]
\newtheorem{theorem}[lemma]{Theorem}
\newtheorem{rmk}{Remark}[section]

\title{Characterization of groups $E_6(3)$ and ${^2}E_6(3)$ by Gruenberg--Kegel graph\footnote{The work is supported by the Mathematical Center in Akademgorodok under the agreement No. 075-15-2019-1675 with the Ministry of Science and Higher Education of the Russian Federation.}}\author{A.\,P.~Khramova, N.\,V.~Maslova, V.\,V.~Panshin, and A.\,M.~Staroletov}
\date{}

\begin{document}
\maketitle
\vspace{-25px}
\begin{abstract}
The Gruenberg--Kegel graph (or the prime graph) $\Gamma(G)$ of a finite group $G$ is defined as follows. The vertex set of $\Gamma(G)$ is the set of all prime divisors of the order of $G$. Two distinct primes $r$ and $s$ regarded as vertices are adjacent in $\Gamma(G)$ if and only if there exists an element of order $rs$ in $G$.
Suppose that $L\cong E_6(3)$ or $L\cong{}^2E_6(3)$. We prove that if $G$ is a finite group such that $\Gamma(G)=\Gamma(L)$, then $G\cong L$.
\end{abstract}

\section{Introduction}
Given a finite group $G$, denote by $\omega(G)$ the spectrum of $G$, that is the set of all its element orders. The set of all prime divisors of the order of $G$ is denoted by $\pi(G)$. {\it The  Gruenberg--Kegel graph} (or {\it the prime graph}) $\Gamma(G)$ of $G$ is defined as follows. The vertex set is the set $\pi(G)$. Two distinct primes $r$ and $s$ regarded as vertices of $\Gamma(G)$ are adjacent in $\Gamma(G)$ if and only if $rs\in\omega(G)$. 
The concept of prime graph of a finite group was introduced by G.K Gruenberg and O. Kegel.
Now this graph is known as Gruenberg-Kegel graph. They also gave a characterization of finite groups with disconnected prime graph but did not publish it. This result can be found in \cite{Will81}, where J.S. Williams started the classification of finite simple groups with disconnected Gruenberg--Kegel graph.

Denote the set of orders of maximal abelian subgroups of $G$ by $M(G)$.
Note that if $\omega(G)=\omega(H)$ or $M(G)=M(H)$,  then $\Gamma(G)=\Gamma(H)$. Consider alternating groups $Alt_5$ and $Alt_6$ of degrees 5 and 6, respectively. Then $\Gamma(Alt_5)=\Gamma(Alt_6)$ but $\omega(Alt_5)=\{1,2,3,5\}=\omega(Alt_6)\setminus\{4\}$ and $M(Alt_5)=\{1,2,3,4,5\}=M(Alt_6)\setminus\{9\}$.

We say that a finite group $G$ is {\it recognizable by $\Gamma(G)$} ($\omega(G)$ or $M(G)$) if for every finite group $H$ the equality $\Gamma(H)=\Gamma(G)$ ($\omega(H)=\omega(G)$ or $M(H)=M(G)$, respectively) implies that $H$ is isomorphic to $G$.
Clearly, if $G$ is recognizable by $\Gamma(G)$, then it is also recognizable by $\omega(G)$ and $M(G)$. The converse is not true in general: the group $Alt_5$ is known to be uniquely determined by spectrum \cite{Shi}, while there are infinitely many groups with the same spectrum as $Alt_6$ \cite{Brandl}. The modern  state of the study on characterization of simple groups by Gruenberg--Kegel grpah can be found, for example, in the recent work by P.~J.~Cameron and the second author~\cite{CamMas}. In particular, in~\cite{CamMas} the authors have proved that if a finite group $L$ is recognizable by $\Gamma(L)$,
then $L$ is almost simple, that is its socle is a nonabelian simple group.

If $p$ is a prime and $q=p^k$ is its power, then by $E_6^+(q)$ and $E_6^-(q)$ we denote the simple exceptional groups $E_6(q)$ and $^2{}E_6(q)$, respectively. Finite groups $G$ such that $\omega(G)=\omega(L)$, where $L\cong E^{\pm}_6(q)$, were described in \cite{Kon07,Grech15,Zvez16}, in particular, if $p\in \{2,11\}$, then the equality $\omega(G)=\omega(L)$ implies $G\cong L$. In \cite{Khos} it is proved that if $G$ is a finite group with $M(G)=M(L)$, where $L\cong E_6(q)$, then $G$ has a unique nonabelian composition factor and this factor is isomorphic to $L$. Nevertheless, there are few results about groups having Gruenberg--Kegel graph as simple groups $E^\pm_6(q)$. In \cite{KonE6} and \cite{Kon2E6} it is proved that if $L\cong E^\pm_6(2)$ and $\Gamma(G)\cong\Gamma(L)$, then $G\cong L$.
The purpose of this paper is to show that groups $E^+_6(3)$ and $E^-_6(3)$ are recognizable by their Gruenberg--Kegel graphs. 

We prove the following theorem.

\begin{theorem}
Suppose that $L\cong E^\varepsilon_6(3)$ with $\varepsilon\in\{+,-\}$.
If $G$ is a finite group such that $\Gamma(G)=\Gamma(L)$, then $G\cong E^\varepsilon_6(q)$.
\end{theorem}

\begin{rmk}
This result was obtained during The Great Mathematical Workshop {\rm \cite{bmm}}.
\end{rmk}

\section{Preliminaries}
Recall that a subset of vertices of a graph is called a {\it coclique} if every two vertices of this subset are nonadjacent. Suppose that $G$ is a finite group. Denote by $t(G)$ the maximal size of a coclique in $\Gamma(G)$. If $2\in\pi(G)$, then $t(2,G)$ denotes the maximal size of a coclique containing vertex $2$ in $\Gamma(G)$. 	
\begin{lemma}[\cite{thvasilev}]\label{vas}
Suppose that $G$ is a finite group with $t(G)\geq3$ and $t(2,G)\geq2$. Then the following statements hold.
\begin{enumerate}
\item There exists a nonabelian simple group $S$ such that $S\trianglelefteq\overline{G} = G/K\leq\operatorname{Aut}(S)$, where $K$ is the solvable radical of $G$.
\item For every cocliue $\rho$ of $\Gamma(G)$ such that $|\rho|\geq3$, at most one prime of $\rho$ divides $|K|\cdot|\overline{G}/S|$. In particular, $t(S)\geq t(G)-1$.
\item One of the following two conditions holds:
\begin{itemize}
\item every prime $p\in\pi(G)$ nonadjacent to $2$ in $\Gamma(G)$ does not divide $|K|\cdot|\overline{G}/S|$. In particular, $t(2,S)\geq t(2,G)$;
\item $S\cong A_7$ or $L_2(q)$ for some odd $q$, and $t(S)=t(2,S)=3$.

\end{itemize}
\end{enumerate}
\end{lemma}
\begin{lemma}{\rm\cite[Lemma 1]{semid}}\label{semid}
Suppose that $N$ is a normal elementary abelian subgroup of a finite group $G$ and  $H=G/N$. 
Define an automorphism $\phi: H\rightarrow$ Aut$(N)$ as follows: $n^{\phi(gN)}=n^g$. Then $\Gamma(G)=\Gamma(N\rtimes_{\phi}H)$.
\end{lemma}

\begin{lemma}\label{graph}
Suppose that $L=E^\varepsilon_6(q)$, where $\varepsilon\in\{+,-\}$.
Then the following statements hold.
\begin{enumerate}
\item[$(1)$] $\pi(E^{-}_6(3))=\{2,3,5,7,13,19,37,41,61,73\}$ and $\Gamma(E^{-}_6(3))$ is the following{\rm:}
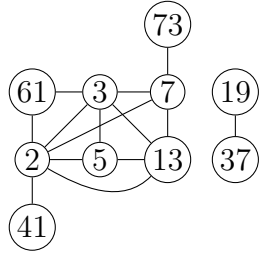
\centering{
\begin{tikzpicture}[node distance={9mm}, main/.style = {draw, circle, fill=white,minimum size=3pt, inner sep=1.2pt}]

\node[main] (2) {$2$};
\node[main] (41) [below of=2] {$41$};
\node[main] (61) [above of=2] {$61$};
\node[main] (3) [right of=61] {$3$};
\node[main] (5) [below  of=3] {$5$};
\node[main] (7) [right of=3] {$7$};
\node[main] (73) [above of=7] {$73$};
\node[main] (13) [below of=7] {$13$};
\node[main] (19) [right of=7] {$19$};
\node[main] (37) [right of=13] {$37$};
\draw (19)--(37);
\draw (2)--(41);
\draw (2)--(5);
\draw (2) to [out=-30,in=-135,looseness=1] (13);
\draw (2)--(61);
\draw (2)--(3);
\draw (2)--(7);
\draw (5)--(13);
\draw (13)--(7);
\draw (7)--(3);
\draw (61)--(3);
\draw (73)--(7);
\draw (3)--(5);
\draw (3)--(13);
\end{tikzpicture}
}
\item[$(2)$] $\pi(E^{+}_6(3))=\{2,3,5,7,11,13,41,73,757\}$ and $\Gamma(E^{+}_6(3))$ is the following{\rm:}
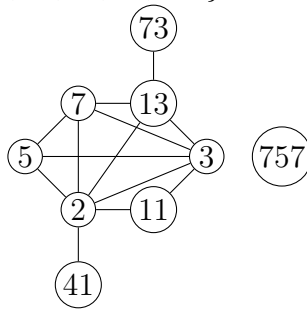
\centering{
\begin{tikzpicture}[node distance={10mm}, main/.style = {draw, circle, fill=white,minimum size=3pt, inner sep=1.2pt}]

\node[main] (2) {$2$};
\node[main] (5) [above left of=2] {$5$};
\node[main] (41) [below of=2] {$41$};
\node[main] (7) [above right of=5] {$7$};
\node[main] (13) [right of=7] {$13$};
\node[main] (73) [above of=13] {$73$};
\node[main] (3) [below right of=13] {$3$};
\node[main] (11) [below left of=3] {$11$};
\node[main] (757) [right of=3] {$757$};

\draw (2)--(7);
\draw (2)--(5);
\draw (2)--(13);
\draw (2)--(11);
\draw (2)--(3);
\draw (2)--(41);
\draw (13)--(73);
\draw (13)--(7);
\draw (3)--(13);
\draw (7)--(5);
\draw (3)--(5);
\draw (3)--(7);
\draw (3)--(11);

\end{tikzpicture}
}
\end{enumerate}
\end{lemma}
\begin{proof}
Apply a criterion of adjacency of two vertices in $\Gamma(E_6^{\pm}(q))$ \cite[Prop~2.5, Prop~3.2, and Prop~4.5]{VasVd}.
\end{proof}

\section{Proof of Theorem}	

Consider a finite group $G$ such that $\Gamma(G)=\Gamma(L)$, where $L=E^\varepsilon_6(3)$ with $\varepsilon\in\{-,+\}$.
Using Lemma~\ref{graph}, we find that $t(G)=t(L)=5$ and $t(2,G)=t(2,L)=3$. It follows from Lemma~\ref{vas} that $S\trianglelefteq\overline{G}=G/K\leq\operatorname{Aut}(S)$, where $K$ is the solvable radical of $G$ and $S$ is a nonabelian simple group.
Denote by $\Omega$ the set of primes in $\pi(L)$ nonadjacent to 2 in $\Gamma(L)$.
Then $\Omega=\{19,37,73\}$ if $\varepsilon=-$ and $\Omega=\{73, 757\}$ if $\varepsilon=+$.
Lemma~\ref{vas} implies that $t(S)\geq4$, and primes from $\Omega$ belong to $\pi(S)$ and do not divide $|G|/|S|$. The proof of the theorem is split into several lemmas.

\begin{lemma} $K$ is nilpotent.
\begin{proof} Consider the action of $G$ on $K$ by conjugation.
Denote $r=37$ if $\varepsilon=-$ and $r=757$ if $\varepsilon=+$. Then $r\not\in\pi(K)$.
Take an element $x\in G$ of order $r$. Since $r$ is nonadjacent to all the vertices of
$\pi(K)$, the action of $x$ on $K$ is fixed-point free. By Thompson's theorem, $K$ is nilpotent.
\end{proof}	
\end{lemma}

\begin{lemma}\label{l:S=2E6}
If $\varepsilon=-$, then $S\cong L$.
\end{lemma}
\begin{proof}		
Observe that 73 is the largest prime in $\pi(S)$.
Inspecting groups from \cite[Table 1]{zav2}, we find that
$U_4(27)$ and $E^{-}_6(3)$ are the only simple groups whose order is divisible by $19\cdot37\cdot73$
and is not divisible by primes greater than 73.
According to \cite[Table 2]{vdovin}, $t(U_4(27))=3$ and hence $S\cong E^{-}_6(3)$, as claimed.
\end{proof}	

\begin{lemma} If $\varepsilon=+$, then $S\cong L$.
\end{lemma}
\begin{proof}	
Note that 757 is the largest prime in $\pi(S)$.
By \cite[Table~3]{zav2}, $S$ is either an alternating group of degree $n\geq757$,
or $L_2(757)$, or a group from the following list:
\begin{multline*}
L_3(27), L_4(27), L_2(3^9), G_2(27), L_2(757^2), S_4(757), E_6(3), L_3(3^6), S_6(27), O_7(27), \\ O^{+}_8(27),U_6(27).
\end{multline*}

If $S$ is an alternating group of degree at least 757, then $17\in\pi(S)\setminus\pi(G)$.
In other cases if $S\not\cong E_6(3)$, then $t(S)\leq3$ according to \cite[Tables~2-4]{vdovin}.
Therefore, $S\cong E_6(3)$, as claimed.
\end{proof}		

\begin{lemma}\label{almost}
$G/K\cong L$.
\end{lemma}
\begin{proof}
Note that $|\operatorname{Aut}(L):L|=2$, so either $G/K\cong L$ or
$G/K\cong\operatorname{Aut}(L)$. Suppose that $G/K\cong\operatorname{Aut}(L)$.
Let $\gamma$ be a graph automorphism of order 2 of $L$. By~\cite[Proposition~4.9.2.]{GLS},
we have $C_L(\gamma)\cong F_4(3)$. Since $73\in\pi(F_4(3))$ and vertices 2 and 73 are nonadjacent in $\Gamma(G)$, we arrive at a contradiction.
\end{proof}

\begin{lemma}\label{l:pi3-7}
If $\varepsilon=-$, then $\pi(K)\subseteq \{3,7\}$
and if $\varepsilon=+$, then $\pi(K)\subseteq \{3,13\}$.
\end{lemma}
\begin{proof}
Take any prime $p\in\pi(K)$.
Since $K$ is nilpotent, we can assume that $K$ is a $p$-group.
Factoring $G$ by $\Phi(K)$, we arrive at a situation where $K$ is an elementary abelian $p$-group. According to~\cite[Table 5.1]{sylow5}, we see that ${}^3D_4(3)\le G/K$. Consider the action of ${}^3D_4(3)$ on $K$ defined by $\phi$ as in Lemma~\ref{semid}. Take an element $g\in{}^3D_4(3)$ of order $73$. If $p\neq3$, then $g$ fixes an element in $K$ by~\cite[Proposition 2]{zav1} and hence 73 and $p$ are adjacent in $\Gamma(G)$. Lemma~\ref{graph} implies that $p\in\{3,7\}$ if
$\varepsilon=-$ and $p\in\{3,13\}$ if
$\varepsilon=+$, as claimed.	
\end{proof}

\begin{lemma}\label{l:K}
$\pi(K)\subseteq \{3\}$.
\begin{proof}
Suppose that $7\in\pi(K)$ or $13\in\pi(K)$. Factoring $G$ by $\Phi(K)$, we arrive at a situation where $K$ is an elementary abelian group. By Lemma~\ref{almost}, we have $G/K\cong L$.

According to ~\cite[Table 5.1]{sylow5}, we see that $P\Omega^{+}_{8}(3)<L$.
Comparing orders of $L$ and $P\Omega^{+}_{8}(3)$, we infer that their Sylow 5-subgroups are isomorphic. Therefore, Sylow 5-subgroups of $L$ are non-cyclic \cite[Table 3]{tori}.
Consider a Sylow 5-subgroup $P$ of $L$. Denote by $\widetilde{P}$ the full preimage of $P$ in $G$. 
The conjugation action of 5-elements of $\widetilde{P}$ on $K$ is fixed-point free, 
so $\widetilde{P}$ is a Frobenious group. It follows from \cite[Chap.~10, Theorem~3.1~(iv)]{Gor} that $P$ is cyclic; a contradiction.
\end{proof}
\end{lemma}	
\begin{lemma}\label{l:K=1}
$K=1$.
\end{lemma}
\begin{proof} 		
By Lemma~\ref{l:K}, $K$ is a 3-group. Assume that $K\neq1$. As above, we can assume that 
$K$ is elementary abelian. 
According to~\cite[Table 5.1]{sylow5}, we see that $F_4(3)\le G/K$. Consider the action of $F_4(3)$ on $K$ as in Lemma~\ref{semid}. Since $F_4(3)$ is unisingular~\cite[Theorem~1.3]{unising}, any element of order 73 in $F_4(3)$ fixes
some non-identity element in $K$. Therefore, primes 3 and 73 are adjacent in $\Gamma(G)$; a contradiction.
\end{proof}

Lemma~\ref{l:K=1} implies that $G\cong L$. This completes the proof of Theorem.

\noindent AUTHORS:

\medskip

\noindent Antonina~P.~Khramova (Sobolev Institute of Mathematics, Novosibirsk),  akhramova@math.nsc.ru

\medskip

\noindent Natalia~V.~Maslova (Krasovskii Institute of Mathematics and Mechanics UB RAS, Ural Federal University, and Ural Mathematical Center, Yekaterinburg),  butterson@mail.ru

\medskip

\noindent Viktor~V.~Panshin (Sobolev Institute of Mathematics and Novosibirsk State University, Novosibirsk), v.pansh1n@yandex.ru

\medskip

\noindent Alexey~M.~Staroletov (Sobolev Institute of Mathematics and Novosibirsk State University, Novosibirsk), staroletov@math.nsc.ru


\begin{thebibliography}{9}

\bibitem{Will81} J.S.~Williams, {\it Prime graph components of finite groups}, J. Algebra, \textbf{69}:2 (1981), 487--513.

\bibitem{Shi}
W.~Shi, {\it A characteristic property of $A_5$}, Journal of Southwest-China Teachers' University, \textbf{2} (1986), 11--14 (in Chinese).

\bibitem{Brandl} R.\,Brandl, W.J.\,Shi, {\it Finite groups whose element orders are consecutive integers},  J. Algebra, \textbf{143}:2 (1991), 388--400. 

\bibitem{CamMas}
P.J.~Cameron, N.V.~Maslova, {\it Criterion of unrecognizability of a finite group by its Gruenberg--Kegel graph}, J. Algebra, to appear. \href{https://doi.org/10.1016/j.jalgebra.2021.12.005}{DOI:10.1016/j.jalgebra.2021.12.005}

\bibitem{Kon07} A.S.~Kondrat$'$ev, {\it Quasirecognition by the set of element orders of the groups
$E_6(q)$ and ${^2}E_6(q)$}, Sib. Math. J., \textbf{48}:6 (2007), 1001--1018. 

\bibitem{Grech15} M.A.~Grechkoseeva, {\it On element orders in covers of finite simple groups of Lie type}, J. Algebra Appl., \textbf{14}:4 (2015), 1550056. 

\bibitem{Zvez16} M.A.~Zvezdina, {\it Spectra of automorphic extensions of finite simple exceptional groups of Lie type}, Algebra and Logic, \textbf{55}:5 (2016), 354--366. 

\bibitem{Khos} Z.~Momen, B.~Khosravi, {\it Quasirecognition of $E_6(q)$ by the orders of maximal abelian subgroups}, J. Algebra Appl., \textbf{17}:7 (2018), 1850122. 

\bibitem{KonE6} W.~Guo, A.S.~Kondrat'ev, N.V.~Maslova, {\it Recognition of the group $E_6(2)$ by Gruenberg--Kegel graph}, Trudy Inst. Mat. i Mekh. UrO RAN, \textbf{27}:4 (2021), 263--268.

\bibitem{Kon2E6}  A.S. Kondrat$'$ev, {\it Recognizability by prime graph of the group $^2E_6(2)$}, Fundam. Prikl. Mat., \textbf{22}:5 (2019), 115--120 (in Russian). 

\bibitem{bmm} The Great Mathematical Workshop, July 12--17 and August 16--21, 2021 with an intermodule work in between, \url{http://mca.nsu.ru/bmm\_english/}.

\bibitem{thvasilev}
A.V.~Vasil$'$ev, {\it On connection between the structure of finite group and properties of its prime graph}, Sib. Math. \textbf{46}:3 (2005), 396--404. 

\bibitem{semid}
A.M.~Staroletov, {\it On recognition of alternating groups by prime graph}, Sib. Electron. Math. Reports, \textbf{14} (2017), 994--1010. 

\bibitem{VasVd}
A.V.~Vasil$'$ev, E.P.~Vdovin E.P, {\it An adjacency criterion for the prime graph of a finite simple group}, Algebra and Logic, \textbf{44}:6 (2005), 381--406. 

\bibitem{zav2}
A.V.~Zavarnitsine, {\it Finite simple groups with narrow prime spectrum}, Sib. Electron. Math. Reports, \textbf{6} (2009), 1--12. 

\bibitem{vdovin}
A.V.~Vasil$'$ev, E.P.~Vdovin, {\it Cocliques of maximal size in the prime graph of a finite simple group, Algebra and Logic}, \textbf{50}:4 (2011), 425--470. 

\bibitem{GLS}
D.~Gorenstein, R.~Lyons, R.~Solomon, {\it The classification of the finite simple groups}. Number 3. Part I. Chapter A, Mathematical Surveys and Monographs, vol. \textbf{40}, American Mathematical Society, Providence, RI, (1998). 

\bibitem{sylow5}
M.W.~Liebeck, J.~Saxl, G.M.~Seitz, {\it Subgroups of maximal rank in finite exceptional groups of Lie type}, Proc. London Math. Soc. (3) \textbf{65}:2 (1992), 297--325. 

\bibitem{zav1}
A.V.~Zavarnitsine, {\it Finite groups with a five-component prime graph, Sib. Math. J.}, \textbf{54}:1 (2013), 40--46.  

\bibitem{tori}
A.A.~Buturlakin, M.A.~Grechkoseeva, {\it The cyclic structure of maximal tori of the finite classical groups}, Algebra and Logic, \textbf{46}:2 (2007), 73--89. 

\bibitem{Gor}
D.~Gorenstein, {\it Finite groups}, Harper and Row, New York, 1968.

\bibitem{unising}
R.~Guralnick, P.H.~Tiep, {\it Finite simple unisingular groups of Lie type}, J.~Group Theory, \textbf{6}:3 (2003), 271--310. Zbl 1046.20013.

\end{thebibliography}
\end{document}